\documentclass[11pt]{article}
\usepackage{amsthm}
\usepackage{amsfonts}
\usepackage{amsmath}
\usepackage{amssymb}
\usepackage{colonequals}
\usepackage{epsfig}
\usepackage{url}
\newcommand{\GG}{{\cal G}}

\newcommand{\col}{\text{col}}
\newcommand{\vol}{\text{vol}}

\newcommand{\mc}[1]{\mathcal{#1}}

\newcommand{\bb}[1]{\mathbb{#1}}

\newtheorem{theorem}{Theorem}
\newtheorem{corollary}[theorem]{Corollary}
\newtheorem{lemma}[theorem]{Lemma}
\newtheorem{example}[theorem]{Example}

\newtheorem{observation}[theorem]{Observation}
\newtheorem{question}[theorem]{Question}

\title{Sublinear separators in intersection graphs of convex shapes}
\author{Zden\v{e}k Dvo\v{r}\'ak\thanks{Computer Science Institute, Charles University, Prague, Czech Republic. E-mail: {\tt rakdver@iuuk.mff.cuni.cz}.
Supported by project 17-04611S (Ramsey-like aspects of graph coloring) of Czech Science Foundation.}\and
Rose McCarty\thanks{Department of Combinatorics and Optimization, University of Waterloo. Email: {\tt rose.mccarty@uwaterloo.ca}.}
\and Sergey Norin\thanks{Department of Mathematics and Statistics, McGill University. Email: {\tt snorin@math.mcgill.ca}. Supported by an NSERC grant 418520.}}
\date{}

\begin{document}
\maketitle

\begin{abstract}
We give a natural sufficient condition for an intersection graph of compact convex sets in $\mathbb{R}^d$
to have a balanced separator of sublinear size.  This condition generalizes several previous results
on sublinear separators in intersection graphs.  Furthermore, the argument used to prove the existence
of sublinear separators is based on a connection with generalized coloring numbers which has not been
previously explored in geometric settings.
\end{abstract}

\section{Introduction}

For a graph $G$, a set $X\subseteq V(G)$ is a \emph{balanced separator} if every component of $G-X$ has at most $\tfrac{2}{3}|V(G)|$ vertices.  For a function $f:\mathbb{N}\to\mathbb{N}$,
we say that $G$ has \emph{$f$-separators} if every subgraph $H$ of $G$ has a balanced separator of size at most $f(|V(H)|)$.  We say that a class $\GG$ of graphs
has \emph{sublinear separators} if there exists a sublinear function $f$ such that every graph in $\GG$ has $f$-separators.

Famously, Lipton and Tarjan~\cite{lt79} proved that planar graphs have $O(\sqrt{n})$-separators, and demonstrated how the corresponding natural
recursive decomposition on small balanced separators can be used in design of divide-and-conquer style algorithms for various problems~\cite{lt80}.
These results motivated further investigation of graph classes with sublinear separators.  Gilbert et al.~\cite{gilbert} extended the result of Lipton
and Tarjan to graphs drawn in any fixed surface, and Alon et al.~\cite{alon1990separator} proved the substantially more general result that for any graph $H$,
the graphs avoiding $H$ as a minor have $O(\sqrt{n})$-separators.  Even more generally, graphs in which for all $d$, all minors of depth $d$ have density bounded by a polynomial of $d$,
have sublinear separators~\cite{plotkin,dvorak2016strongly,biswal2010eigenvalue}.
Sublinear separators also naturally arise in geometric settings, such as finite element meshes and overlap graphs~\cite{miller1998geometric},
graphs embedded with bounded distortion of distances and their generalizations, and nearest-neighbor graphs~\cite{miller1997separators},
and intersection graphs of connected subsets of a surface with subquadratic number of edges~\cite{lee2016separators}.

Many of these geometric results can be formulated in terms of intersection graphs.
An \emph{intersection representation} of a graph $G$ in $\mathbb{R}^d$ is a function $\varphi$ that to every vertex of $G$ assigns a non-empty set $\varphi(v)\subseteq\mathbb{R}^d$
such that for all distinct $u,v\in V(G)$, the sets $\varphi(u)$ and $\varphi(v)$ intersect if and only if $uv\in E(G)$.
Of course, without further constraints, any graph can be represented in this way, and in particular, the
existence of an intersection representation cannot imply sublinear separators.  For example, every clique
has an intersection representation obtained by choosing the same set to represent each vertex.  To avoid this issue, it is natural to require
that each point is contained only in a bounded number of the sets $\varphi(v)$ for $v\in V(G)$.
For an integer $c$, we say that a representation $\varphi$ of a graph $G$ is \emph{$c$-thin} if for every point $x$ of the space, there exist at most $c$ vertices $v\in V(G)$ such that $x\in\varphi(v)$.

While not all graphs admitting a $c$-thin intersection representation in bounded dimension have sublinear separators, there are many
results giving additional constraints sufficient to enforce this property.  We are particularly interested in constraints which are given only in terms of the ``shapes'' of the sets representing the vertices (their properties invariant on translation), rather than constraints on the positions of the shapes or their intersections; these constraints also directly apply to all induced subgraphs of the represented graph,
which is useful in applications where the graph is recursively decomposed on the small balanced separators.

A set in a Euclidean space is \emph{non-degenerate} if it is not contained in any proper affine subspace (and in particular, it is non-empty).
A set is a \emph{shape} if it is non-degenerate and compact.
Let us remark that the assumption of non-degeneracy is just a technicality to simplify some of our statements and arguments; allowing degenerate shapes would not qualitatively affect our results.
For a set $S$ of shapes, an \emph{$S$-shaped intersection representation} of a graph $G$ is an intersection representation $\varphi$ that to every vertex
of $G$ assigns a translation of one of the shapes in $S$.  Note that the translations of a single shape can be used to represent several different vertices of $G$.

We are interested in describing sets $S$ of shapes in $\bb{R}^d$ such that graphs with a $c$-thin $S$-shaped intersection representation admit sublinear separators. A prime example of such a set is the set $\mc{B}$ of balls in $\mathbb{R}^d$. Any graph with a $c$-thin $\mc{B}$-shaped intersection representation has $O(n^{1-1/d})$-separators.
This result now has a number of distinct (although related) proofs~\cite{miller1997separators,smith1998geometric,har2015approximation}, and is motivated by Koebe's result~\cite{koebe}
implying that planar graphs have a $2$-thin intersection representation by balls in $\mathbb{R}^2$; in particular, this gives an alternative proof that planar graphs have $O(\sqrt{n})$-separators.
The proofs in~\cite{smith1998geometric,har2015approximation} also apply in a more general setting where all shapes in $S$ are connected $k$-fat sets, i.e.,
sets $Q$ such that every ball $B$ with the center in $Q$ satisfies $\vol(B\cap Q)\ge \min(\vol(B)/k,\vol(Q))$; here $\vol$ denotes the volume in $\mathbb{R}^d$.

We are particularly interested in the case where the sets in $S$ are convex.  Indeed, non-convexity
is somewhat troublesome, as arbitrarily large cliques can be $2$-thinly represented by translations of the same non-convex set (e.g., an $L$-shape in $\mathbb{R}^2$, as depicted in Figure~\ref{fig:Ls}).
On the other hand, every compact convex set can be transformed via a bijective affine transformation
to a fat set (as is easily seen using a result of~\cite{chakerian1967some}, stated as Lemma~\ref{lemma-envel} below) and thus the aforementioned results
of~\cite{smith1998geometric,har2015approximation} show that if elements of $S$ are obtained from one convex shape in $\mathbb{R}^d$ by homothety, then any graph with a $c$-thin $S$-shaped representation
has $O(n^{1-1/d})$-separators.  More generally, these results also apply in the case that all shapes in $S$
are convex and have aspect ratio (the ratio of their diameter to their height) bounded by a constant, up to homothety.

\begin{figure}
\centering
\includegraphics[scale=.75]{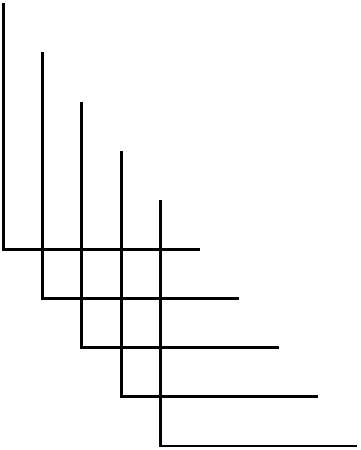}
\caption{An intersection representation of $K_5$ by an $L$-shape.}
\label{fig:Ls}
\end{figure}

However, we clearly cannot allow $S$ to be the set of all convex shapes in $\mathbb{R}^d$.  For example, if $S$ consists of all line segments in $\mathbb{R}^2$ with two different
directions, then we can use their translations to obtain a $2$-thin representation of an arbitrary complete bipartite graph, as in Figure~\ref{fig:narrowR} (strictly speaking, elements of $S$ as given are not
non-degenerate, and thus they are not shapes; however, this issue can be easily overcome by considering narrow rectangles instead of line segments).
This is a special case of a more general issue.  For a set $B\subseteq\mathbb{R}^d$ and a positive real number $k$, let $kB=\{kx:x\in B\}$.
For sets $B_1, B_2\subseteq\mathbb{R}^d$, we write $B_1\le_k B_2$ if a translation of $B_1$ is a subset of $kB_2$.
We say that $B_1$ and $B_2$ are \emph{$\le_k$-comparable} if $B_1\le_k B_2$ or $B_2\le_k B_1$, and \emph{$\le_k$-incomparable} otherwise.
In Lemma~\ref{lemma-incomp}, we will show that if $S$ contains two $\le_k$-incomparable shapes for a large $k$,
then there exists a $2$-thin $S$-shaped intersection representation of a large complete bipartite graph.

\begin{figure}
\centering
\includegraphics[scale=.75]{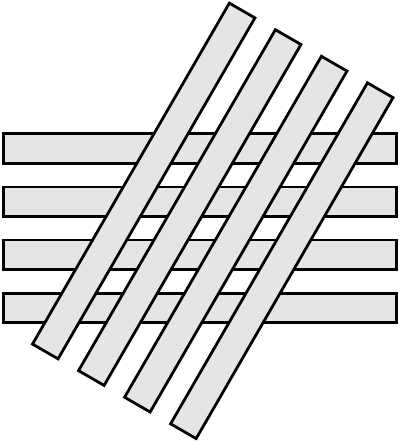}
\caption{An intersection representation of $K_{4,4}$ by narrow rectangles with two different directions.}
\label{fig:narrowR}
\end{figure}

Hence, we need to assume that every two shapes in $S$ are $\le_k$-comparable, for some fixed constant $k$.  This still turns out not to be enough to exclude the existence of
a $2$-thin representation of an arbitrarily large complete bipartite graph. For a depiction of the next example, see Figure~\ref{fig:notSq}.
\begin{example}\label{ex-narrow}
For $i=1,\ldots, m$, let $L_i$ be (a narrow box around) the line segment $\{(i,y,0):1\le y\le m\}$ in $\mathbb{R}^3$.
Let $R_0=L_1$, and for $i=1,\ldots,m$, let $R_i$ be the convex closure of $\{(x,i,0):1\le x\le m\}\cup (R_{i-1} + v_i)$ for a vector $v_i=(0,y_i,1)$ selected by choosing $y_i>0$
so that $R_i$ is disjoint from $R_1\cup\ldots\cup R_{i-1}$; this is possible, since by choosing $y_i$ sufficiently large, we can fit the ``wedge'' $R_i$ into the angle between $R_{i-1}$
and the $xy$-plane.  This gives a $2$-thin representation of $K_{m,m}$ in which any two shapes are $\le_1$-comparable.
\end{example}

\begin{figure}
\centering
\includegraphics[scale=.75]{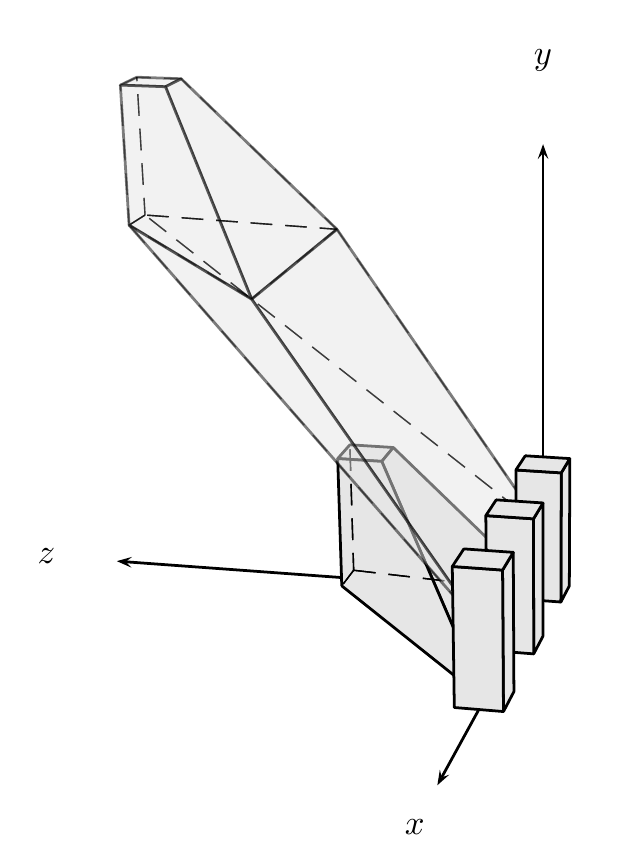}
\caption{An intersection representation of $K_{2,3}$ by $R_1$, $R_2$, $L_1$, $L_2$, and $L_3$.}
\label{fig:notSq}
\end{figure}

The issue is that the copy of $R_{i-1}$ in $R_i$ showing
that $R_{i-1}\le_1R_i$ is too far from the line $\{(x,i,0):1\le x\le m\}$ which actually participates in intersections.
There are two natural ways we could overcome this issue: we could for every point $x\in R_i$ require that either
\begin{itemize}
\item there exists a translation of $R_{i-1}$ contained in (a multiple of) $R_i$ and placed relatively near to $x$, or
\item there exists a translation of $R_{i-1}$ containing $x$ which has a relatively large intersection with $R_i$.
\end{itemize}
This motivates the following definitions.  For sets $B_1, B_2\subseteq\mathbb{R}^d$ and real numbers $k,s\ge 1$, we write
\begin{itemize}
\item $B_1\le_{k,s} B_2$ if for every point $x\in kB_2$, there exists a translation $B'_1$ of $B_1$ and a translation $B''_1$ of
$sB_1$ such that $x\in B''_1$ and $B'_1\subseteq B''_1\cap kB_2$, and
\item $B_1\sqsubseteq_s B_2$ if for every point $x\in B_2$, there exists
a translation $B'_1$ of $B_1$ such that $x\in B'_1$ and $\vol(B'_1\cap B_2)\ge\tfrac{1}{s}\vol(B_1)$.
\end{itemize}
As we will see in Lemmas~\ref{lemma-rel1} and \ref{lemma-rel2}, both of these notions are equivalent up to a transformation of the parameters.

We say that $B_1$ and $B_2$ are \emph{$\sqsubseteq_s$-comparable} if $B_1\sqsubseteq_s B_2$ or $B_2\sqsubseteq_s B_1$. For brevity we combine the conditions we are planning to impose on the representation, and we say that a representation $\varphi$  of a graph $G$ is \emph{$(c,\sqsubseteq_s)$-tame} if $\varphi$ is $c$-thin, $\varphi(v)$ is a convex shape for every $v \in V(G)$, and  $\varphi(u)$ and $\varphi(v)$ are $\sqsubseteq_s$-comparable for all $u,v \in V(G)$. With these conditions, every graph we obtain turns out to have sublinear separators.

\begin{theorem}\label{thm-inters}
For all real numbers $c,s\ge 1$ and positive integers $d$, there exists a real number $\beta$ such that 
every $n$-vertex graph admitting a $(c,\sqsubseteq_s)$-tame representation in $\bb{R}^d$ has a balanced separator of size at most $\beta n^{1-\frac{1}{2d+4}}$.
\end{theorem}

Let us remark that to our knowledge, Theorem~\ref{thm-inters} does not follow from any of the previous results. In particular, Theorem~\ref{thm-inters} seems to be the first result which allows the shapes to have unbounded aspect ratios (even after a bijective affine transformation). In Section~\ref{sec-AR}, we exhibit a class of graphs with $(4,\sqsubseteq_1)$-tame representations in $\bb{R}^3$ which cannot be represented by convex shapes of bounded aspect ratio in bounded dimension.

Theorem~\ref{thm-inters} also implies all previous results on intersection graphs of convex sets that are stated purely in terms of the shapes.  Note that for every finite collection $S$ of convex shapes there exists $s$ such that shapes in $s$ are pairwise $\sqsubseteq_s$-comparable, and so Theorem~\ref{thm-inters} is applicable to $S$-shaped intersection representations. This is an exception to our claim that non-degeneracy is just a technicality; as we have seen above, arbitrarily large complete bipartite graphs can be represented over just two degenerate shapes (line segments with different directions).

We derive Theorem~\ref{thm-inters} from a stronger statement regarding generalized coloring numbers.
Let $\prec$ be a linear ordering of vertices of a graph $G$.  A path $P$ from a vertex $v$ to a vertex $x$ \emph{respects} the ordering $\prec$
if $v\prec z$ for all $z\in V(P)\setminus\{x,v\}$.
For a vertex $v\in V(G)$ and a non-negative integer $r$,
let $L_{G,\prec,r}(v)$ be the set of vertices $x\in V(G)$ such that $x\preceq v$ and some path from $v$ to $x$ of length at most $r$ respects $\prec$.
Let $\col_{\prec,r}(G)$ be the maximum
of $|L_{G,\prec,r}(v)|$ over $v\in V(G)$, and let the \emph{strong coloring number} $\col_r(G)$ be the minimum of $\col_{\prec,r}(G)$ over all linear orderings $\prec$ of $V(G)$.
A combination of Observation~10 of~\cite{espsublin} and Lemma~2 of~\cite{onsubsep} (based on the result of Plotkin, Rao, and Smith~\cite{plotkin}) gives the following.
\begin{lemma}\label{lemma-sublin}
For all real numbers $a,p>0$, there exists a real number $\gamma$ such that the following claim holds.  If a graph $G$ satisfies $\col_r(G)\le ar^p$ for every positive integer $r$, then
$G$ has a balanced separator of size at most $\gamma n^{1-\frac{1}{2p+4}}$.
\end{lemma}

Theorem~\ref{thm-inters} thus is a consequence of the following claim, showing that a graph $G$ which admits a $(c,\sqsubseteq_s)$-tame representation in $\bb{R}^d$ for fixed $c$, $s$ and $d$ has an ordering $\prec$ such that $\col_{\prec,r}(G)$ is polynomial in $r$; note that the same ordering $\prec$ works uniformly for all
values of $r$.
\begin{theorem}\label{thm-color}
For all real numbers $c,s\ge 1$ and positive integers $d$, there exists exists a real number $\delta$ such that the following claim holds. Let $\varphi$ be a $(c,\sqsubseteq_s)$-tame representation of a graph $G$ in $\bb{R}^d$, and
let $\prec$ be a linear ordering of vertices $v\in V(G)$ in a non-increasing order according to the volume of $\varphi(v)$.
Then $\col_{\prec,r}(G)\le \delta r^d$ for every positive integer $r$.
\end{theorem}

We obtained Theorem~\ref{thm-color} on our quest to find a (rough) geometric
characterization of graphs with polynomially bounded strong coloring numbers,
modeled after the following beautiful result of Krauthgamer and
Lee~\cite{KraLee07}. For a graph $G$, a vertex $v\in V(G)$, and a non-negative integer
$r$, let $B_{G,r}(v)$ denote the set of vertices of $G$ reachable from $v$ by a
path of length at most $r$, and let $\Delta_r(G)=\max_{v \in V(G)}|B_r(v)|$. Thus
$\Delta_r(G)$ generalizes the concept of maximum degree of a graph in the same
way that $\col_r(G)$ generalizes the coloring number (a.k.a. ``degeneracy'').
We say that the \emph{growth} of a graph $G$ is bounded by a function $p:\mathbb{N}\to\mathbb{R}$
if $\Delta_r(G)\le p(r)$ holds for all $r\ge 0$.
  
\begin{theorem}[Krauthgamer and Lee~\cite{KraLee07}]\label{thm-KL}
  For every polynomial $p$, there exists an integer $d$ such that the following holds.
  For every graph $G$ whose growth is bounded by $p$, there exists an injection $\varphi: V(G) \to \bb{Z}^d$ such that
  $\| \varphi(u) - \varphi(v)\|_{\infty}=1$ for every edge $uv \in E(G)$.
\end{theorem}

Note that the graphs admitting an embedding $\varphi$ as in
Theorem~\ref{thm-KL} satisfy $\Delta_r(G) \leq (2r+1)^d$ for every $r$, and
thus Theorem~\ref{thm-KL} characterizes graph classes with polynomial growth
(a class of graphs has \emph{polynomial growth} if there exists a polynomial $p$
such that every graph from the class has growth bounded by $p$).
Also note that representing each vertex $v$ by a ball of radius $1/2$ centered at $\varphi(v)$ gives
a $2$-thin intersection representation by translations of a single convex shape.
Answering the following question would determine whether Theorem~\ref{thm-color} gives a similar characterization. 

\begin{question}\label{que-color}
	Given a polynomial $p$, do there necessarily exist real numbers $c,s\ge 1$ and a positive integer $d$,
	such that the following holds: every graph $G$ satisfying $\col_{r}(G) \le p(r)$ for every non-negative integer $r$ admits
	a $(c,\sqsubseteq_s)$-tame representation in $\bb{R}^d$?
\end{question}

We suspect that the answer to Question~\ref{que-color} is negative based on the following. The map $\varphi$ in Theorem~\ref{thm-KL} can be considered as a product of adjacency preserving maps from $V(G)$ to an infinite path. By analogy one could hope for a representation of graphs with polynomially bounded coloring numbers that has the same product structure, i.e. to have vertices represented by boxes. (A \emph{box} in $\bb{R}^d$ is a product of $d$ closed intervals.)
The following result shows that such a representation does not always exist. 

\begin{theorem}\label{thm-converse}
	For all real  $c,s\ge 1$, every positive integer $d$, and every non-decreasing function $f: \bb{N} \to [3,+\infty)$ such that $\lim_{r\to \infty}f(r) = + \infty$, there exists a graph $G$ such that \begin{description} \item[(i)] there exists a linear order $\prec$ on $V(G)$ such that $\col_{\prec,r}(G)\le f(r)$ for every non-negative integer $r$, and \item [(ii)]
 $G$ does not admit a  $(c,\sqsubseteq_s)$-tame $S$-shaped intersection  representation $\varphi$ in $\bb{R}^d$, where $S$ is the set of boxes. \end{description}
\end{theorem}

The rest of the paper is organized as follows.  In Section~\ref{sec-nocom}, we prove Lemma~\ref{lemma-incomp} showing the necessity of pairwise $\le_k$-comparability.
In Section~\ref{sec-main}, we prove the main result, Theorem~\ref{thm-color}. In Section~\ref{sec-AR} we prove that this theorem allows for classes which cannot be represented by convex shapes of bounded aspect ratio. In Section~\ref{sec-converse}, we  prove Theorem~\ref{thm-converse}. Finally, in Section~\ref{sec-rem} we give further remarks and open problems
regarding the notions we introduced.

\section{Non-comparable shapes}\label{sec-nocom}

We will need the following approximation result for convex shapes.
Let $v_1$, \ldots, $v_d$ be an independent set of vectors in the $d$-dimensional Euclidean space.  A \emph{parallelepiped} with sides $v_1$, \ldots, $v_d$ and center $p$ is
$\{p+\sum_{i=1}^d \alpha_iv_i:-1\le \alpha_i\le 1\}$.  A parallelepiped $T$ is
an \emph{envelope} of a subset $B$ of the Euclidean space if $B\subseteq T$ and $T$ has the smallest volume among
parallelepipeds with this property.  Observe that if $B$ is compact and non-degenerate, then it has a (not necessarily unique) envelope.
The following claim follows by an inspection of the proof of Lemma~1 in~\cite{chakerian1967some}.

\begin{lemma}\label{lemma-envel}
Let $B$ be a convex shape in $\mathbb{R}^d$.  If $T$ is an envelope of $B$, then $T\le_d B$.
\end{lemma}

The \emph{height} of a compact set $B\subseteq \mathbb{R}^d$ is the minimum $h$ such that $B$ is contained between two parallel hyperplanes
at distance $h$.  In particular,
the height of a parallelepiped with sides $Z$ is the minimum over $z\in Z$ of the distance from $z$ to the hyperplane spanned by $Z\setminus\{z\}$.
Note that a parallelepiped of height $h$ and center $v$ contains a ball of diameter $h$, centered at $v$.
This has the following approximate generalization.

\begin{lemma}\label{lemma-height}
Every convex shape $B$ in $\mathbb{R}^d$ of height $h$ contains a ball of diameter $h/d$.
\end{lemma}
\begin{proof}
Let $T$ be an envelope of $B$.  The height of $T$ is at least $h$.  By Lemma~\ref{lemma-envel},
a translation of $\tfrac{1}{d}T$, whose height is at least $h/d$, is contained in $B$.
Consequently, $B$ contains a ball of diameter $h/d$.
\end{proof}

We are now ready to show the necessity of pairwise $\le_k$-comparability.

\begin{lemma}\label{lemma-incomp}
Let $S$ be a set of convex shapes in $\mathbb{R}^d$, let $m$ be a positive integer, and let $k=2d^{5/2}m$. 
If there exist $\le_k$-incomparable shapes $B_1,B_2\in S$, then $K_{m,m}$ has an $S$-shaped $2$-thin intersection representation.
\end{lemma}
\begin{proof}
By applying a bijective affine transformation to the space, we can assume that the unit cube $T_1=[ -1/2,1/2]^d$ is an envelope of $B_1$.
Let $U$ be the ball of diameter $\sqrt{d}$ centered at the origin, so that $B_1\subseteq T_1\subseteq U$
and $\tfrac{1}{\sqrt{d}}U\subseteq T_1$.  Let $T_2$ be an envelope of $B_2$.

Since $B_1\not\le_k B_2$, we have $U\not\le_k B_2$, and thus $B_2$ does not contain a ball of diameter $\sqrt{d}/k$.
By Lemma~\ref{lemma-height}, $B_2$ has height less than $d^{3/2}/k$.
Hence, there exists a vector $v$ of length $d^{3/2}/k$ and a hyperplane $H$ perpendicular to $v$ such that $B_2$ is contained strictly between $H$ and $H+v$.

By Lemma~\ref{lemma-envel}, there exists $p$ such that $p+\tfrac{1}{d}T_1\subseteq B_1$, and thus the ball $U'=p+d^{-3/2}U$ of diameter $1/d$ is a subset of $B_1$.
Let $s$ be the longest line segment contained in $B_2$ and let $r$ be the length of $s$.
Since $B_2\not\le_k B_1$, we have $B_2\not\le_k U'$.
Note that $B_2$ is contained within a ball of radius $r$ centered at an end of $s$, and thus
$r>\tfrac{k}{2d}$.  Consequently, $s$ contains $m$ points $z_1$, \ldots, $z_m$ such that the distances between the points are greater than $\sqrt{d}$.

For $i=1,\ldots, m$, let $L_i=B_1-p+z_i$ and $R_i=B_2+(i-1)v$.  Note that the shapes $R_1$, \ldots, $R_m$ are disjoint by the choice of $v$,
and that the shapes $L_1$, \ldots, $L_m$ are contained in balls $U-p+z_1$, \ldots, $U-p+z_m$ which are disjoint by the choice of $z_1$, \ldots, $z_m$.
Furthermore, for any $i,j\in\{1,\ldots,m\}$, the point $z_i+(j-1)v$ is contained in both $L_i$ (since $L_i$ contains the ball $z_i+U'-p=z_i+d^{-3/2}U$ of radius
$\tfrac{1}{2d}$ centered at $z_i$ and the distance between $z_i$ and $z_i+(j-1)v$ is less than $\tfrac{md^{3/2}}{k}=\tfrac{1}{2d}$) and $R_j$ (since $z_i\in B_2$).
Hence, the shapes $L_1$, \ldots, $L_m$, $R_1$, \ldots, $R_m$ give a $2$-thin intersection representation of $K_{m,m}$.
\end{proof}

\section{Coloring numbers}\label{sec-main}

For the purpose of this paper we say that  $B \subseteq \bb{R}^d$ is \emph{centrally symmetric} if $B=-B$, i.e. $B$ is invariant under reflection across the origin. We start with a trivial observation.

\begin{observation}\label{obs:central}
	Let $A,B \subseteq \bb{R}^d$ be such that $B$ is centrally symmetric and convex,  and let $C$ be the union of all translates $B'$ of $B$ such that $B' \cap A\neq \emptyset$. Then $C \subseteq A+2B$.
\end{observation}

\begin{proof}
	We need to show that for every $u \in \bb{R}^d$ if $A \cap (u+B) \neq \emptyset$ then $u+B \subseteq  A+2B$. The condition $A \cap (u+B) \neq \emptyset$ is equivalent to $u \in A + (-B)$. As $A+(-B) = A+B$ and $B$ is convex, this in turn implies $u+B \subseteq  A+2B$.
\end{proof}

We can now prove the following lemma, which as we will show later implies Theorem~\ref{thm-color}.
\begin{lemma}\label{lemma-color}
For all real numbers $c,k,s\ge 1$ and positive integers $d$, there exists a real number $\delta$ such that the following claim holds.
Let $G$ be a graph, let $\prec$ be a linear ordering of vertices of $G$, and let $\iota$ and $\omega$ be functions assigning shapes in $\mathbb{R}^d$ to vertices of $G$, satisfying
the following conditions:
\begin{itemize}
\item[(a)] for every $v\in V(G)$, $\omega(v)$ is convex and $\iota(v)\subseteq \omega(v)$,
\item[(b)] for every $x\in\mathbb{R}^d$, there exist at most $c$ vertices $v\in V(G)$ such that $x\in \iota(v)$,
\item[(c)] for every $u,v\in V(G)$ such that $u\prec v$, we have $\omega(v)\le_k \omega(u)$ and $\omega(v)\sqsubseteq_s\iota(u)$, and
\item[(d)] for every $uv\in E(G)$, if $u\prec v$, then $\omega(v)\cap \iota(u)\neq\emptyset$.
\end{itemize}
Then $\col_{\prec,r}(G)\le \delta r^d$ for every positive integer $r$.
\end{lemma}
\begin{proof}
Let $\delta=2cs(2k+1)^dd^d$.
Consider a positive integer $r$ and a vertex $v\in V(G)$; we need to argue that $|L_{G,\prec,r}(v)|\le \delta r^d$.
We can assume $|L_{G,\prec,r}(v)|\ge 1$, as otherwise the claim is trivial.
Let the parallelipiped $p(v)+T(v)$ with center $p(v)$ be an envelope of $\omega(v)$,
and let $Y_v=p(v)+(3+2(r-1)k)T(v)$. Note that $T(v)$ is centrally symmetric.

Suppose that a vertex $u\neq v$ belongs to $L_{G,\prec,r}(v)$; hence, there exists a path $x_0x_1 \ldots x_m$
in $G$ for some $m\le r$ such that $x_0=v$, $x_m=u$, and $v\prec x_1, \ldots, x_{m-1}$.
By (c), for $i=1,\ldots,m-1$, and the definition of an envelope, there exists a translation $T_i$ of $kT(v)$ containing $\omega(x_i)$;
let $T_0=p(v)+T(v)$.  By (a) and (d), $T_i$ intersects $T_{i-1}$ for $1\le i\le m-1$ and $T_{m-1}$ intersects $\iota(u)$
in some point $y_u$. Since $\omega(v)\sqsubseteq_s\iota(u)$ by (c), there exists a translation $B_u$ of $\omega(v)$
such that $y_u\in B_u$ and $\vol(B_u\cap \iota(u)) \ge \tfrac{1}{s}\vol(\omega(v))$.
Let $B'_u = B_u\cap \iota(u)$. Repeated application of Observation~\ref{obs:central} implies $B'_u \subseteq Y_v$.

Choose $B'_u\subseteq\iota(u)$ as described in the previous paragraph for each $u\in L_{G,\prec,r}(v)\setminus\{v\}$, and let
$X_v=\bigcup_{u\in L_{G,\prec,r}(v)\setminus\{v\}} B'_u$.
By (b), we have
\begin{align*}
\vol(X_v)&\ge \frac{1}{c}\sum_{u\in L_{G,\prec,r}(v)\setminus\{v\}} \vol(B'_u)\ge \frac{|L_{G,\prec,r}(v)|-1}{cs}\vol(\omega(v))\\
&\ge \frac{|L_{G,\prec,r}(v)|}{2cs}\vol(\omega(v)).
\end{align*}
On the other hand, $X_v \subseteq Y_v$, and by Lemma~\ref{lemma-envel},
$$\vol(X_v) \leq \vol(Y_v) \leq (3+2(r-1)k)^d\vol(T(v)) \le (2k+1)^dd^dr^d\vol(\omega(v)).$$
Therefore, $|L_{G,\prec,r}(v)|\le \delta r^d$, as required.
\end{proof}

In order to prove Lemma~\ref{lemma-color} implies Theorem~\ref{thm-color}, we need to
establish a connection between the relations $\le_{k,s}$ and $\sqsubseteq_s$.

\begin{lemma}\label{lemma-rel1}
Let $B_1$ and $B_2$ be shapes in $\mathbb{R}^d$ and let $k,s\ge 1$ be real numbers.  If
$B_1\le_{k,s} B_2$, then $B_1\sqsubseteq_{\max(k,s)^d} B_2$.
\end{lemma}
\begin{proof}
Consider any point in $B_2$; without loss of generality, we can translate $B_2$ so that this point is the origin $o$.
Note that $o\in kB_2$, and since $B_1\le_{k,s} B_2$, there exists
a translation $B'_1$ of $B_1$ such that $o\in sB'_1$ and $sB'_1\cap kB_2$ contains a translation of $B_1$,
implying that $\vol(sB'_1\cap kB_2)\ge \vol(B_1)$.
Let $m=\max(k,s)$, so that $\tfrac{s}{m}B'_1\subseteq B'_1$ and $\tfrac{k}{m}B_2\subseteq B_2$.
Note that $o\in B'_1$ and $\vol(B'_1\cap B_2)\ge \vol(\tfrac{s}{m}B'_1\cap \tfrac{k}{m}B_2)=\tfrac{1}{m^d}\vol(sB'_1\cap kB_2)\ge \tfrac{1}{m^d}\vol(B_1)$.
This implies $B_1\sqsubseteq_{m^d} B_2$.
\end{proof}

For a positive integer $d$, let $\gamma_d>0$ denote the maximum $(d-1)$-dimensional volume of an intersection
of a hyperplane with a unit cube in $\mathbb{R}^d$.

\begin{lemma}\label{lemma-rel2}
Let $B_1$ and $B_2$ be shapes in $\mathbb{R}^d$ and let $s\ge 1$ be a real number.
Let $k=sd^{d+3/2}\gamma_d$.  If $B_1\sqsubseteq_s B_2$, then $B_1\le_{k,k} B_2$.
\end{lemma}
\begin{proof}
By applying a bijective affine transformation, we can assume that a unit cube $T_1$ is an envelope of $B_1$;
hence, $\vol(B_1)\le 1$, and Lemma~\ref{lemma-envel} implies $\vol(B_1)\ge d^{-d}$.

Consider any point in $kB_2$; without loss of generality, we can translate $B_2$ so that this point is the origin $o$, and thus $o\in B_2$.
Since $B_1\sqsubseteq_s B_2$, by translating $B_1$, we can without loss of generality assume that $o\in B_1$
and $\vol(B_1\cap B_2)\ge \tfrac{1}{s}\vol(B_1)\ge\tfrac{1}{sd^d}$.  Note that any subset of $T_1$ of height $h$
has volume at most $\gamma_dh$, and thus $B_1\cap B_2$ has height at least $\tfrac{1}{sd^d\gamma_d}$.  By Lemma~\ref{lemma-height},
$B_1\cap B_2$ contains a ball of diameter $\tfrac{1}{sd^{d+1}\gamma_d}$.  Since $B_1$ is contained within a ball of diameter $\sqrt{d}$,
we conclude that $sd^{d+3/2}\gamma_d(B_1\cap B_2)=k(B_1\cap B_2)$ contains a translation of $B_1$.  Consequently,
there exists a translation $B'_1$ of $B_1$ such that $B'_1\subseteq kB_2\cap kB_1$, and $o\in kB_1$.
This implies that $B_1\le_{k,k} B_2$.
\end{proof}

Furthermore, we need to relate comparability to the volume of the two bodies.
\begin{lemma}\label{lemma-cmp}
For every real number $s\ge 1$ and positive integer $d$, there exists a real number $s'\ge s$
such that the following claim holds.
Let $B_1$ and $B_2$ be convex shapes in $\mathbb{R}^d$.
If $B_1$ and $B_2$ are $\sqsubseteq_s$-comparable and $\vol(B_1)\le \vol(B_2)$, then $B_1\sqsubseteq_{s'} B_2$.
\end{lemma}
\begin{proof}
Let $k\ge 1$ be as in Lemma~\ref{lemma-rel2}, and let $s'=\max(s,k^d)$.
Since $B_1$ and $B_2$ are $\sqsubseteq_s$-comparable, it suffices to consider the case
that $B_2\sqsubseteq_s B_1$.  By Lemma~\ref{lemma-rel2}, we have $B_2\le_{k,k} B_1$. So $B_2\le_{k} B_1$ as well,
and thus we can assume that $B_2\subseteq kB_1$.  Consider any point of $B_2$, without loss
of generality the origin $o$.  Then $o\in B_1$ and since $B_2\subseteq kB_2$, we have
$\vol(B_1\cap B_2)=k^{-d} \vol(kB_1\cap kB_2)\ge k^{-d}\vol(kB_1\cap B_2)=k^{-d}\vol(B_2)\ge k^{-d}\vol(B_1)$.
This implies $B_1\sqsubseteq_{s'} B_2$.
\end{proof}

Theorem~\ref{thm-color}  now easily follows.

\begin{proof}[Proof of Theorem~\ref{thm-color}]
For every $v\in V(G)$, let $\iota(v)=\omega(v)=\varphi(v)$.  It suffices to verify that the assumptions of Lemma~\ref{lemma-color} hold.
The assumptions (a) and (d) clearly hold, and (b) holds since $\varphi$ is $c$-thin.  If $u\prec v$, then $\vol(\varphi(v))\le\vol(\varphi(u))$,
and thus $\omega(v)=\varphi(v)\sqsubseteq_{s'}\varphi(u)=\iota(u)$ by Lemma~\ref{lemma-cmp}, and 
$\omega(v)=\varphi(v)\le_{k'}\varphi(u)=\omega(u)$ by Lemma~\ref{lemma-rel2}, for real numbers $s', k'\ge 1$ depending only on $c$, $s$, and $d$;
hence, (c) holds.
\end{proof}

\section{Representability and aspect ratio}\label{sec-AR}
Let $G_1 \boxtimes G_2$ denote the strong product of two graphs $G_1$ and $G_2$. For positive integers $r$ and $t$, let $T_r$ denote the star with $r+1$ vertices and $P_t$ the path with $t$ vertices. In this section we prove that $T_r \boxtimes P_t$ has a $(4, \sqsubseteq_1)$-tame representation by boxes in $\bb{R}^3$ while, for all positive integers $d$, $c$, and $k$, and $r$ and $t$ sufficiently large, it does not admit a $c$-thin intersection representation in $\bb{R}^d$ by convex shapes of aspect ratio at most $k$. The representation in the first lemma is depicted in Figure~\ref{fig:PT}.

\begin{lemma}
For all positive integers $r$ and $t$, the graph $T_r \boxtimes P_t$ admits a $(4, \sqsubseteq_1)$-tame $S$-shaped representation, where $S$ is the set of boxes in $\bb{R}^3$.
\end{lemma}
\begin{proof}
Let $S_1$ and $S_2$ be axis-aligned squares in $\bb{R}^2$ of heights $2r$ and $1$, respectively. The graph $T_r$ has a $2$-thin $\{S_1, S_2\}$-shaped intersection representation $\varphi_1$. Similarly, where $I$ denotes the unit interval, the graph $P_t$ has a $2$-thin $\{I\}$-shaped intersection representation $\varphi_2$. It follows that $\varphi_1 \times \varphi_2$ is a $(4, \sqsubseteq_1)$-tame $\{S_1 \times I, S_2 \times I\}$-shaped representation of the graph $T_r \boxtimes P_t$, where $\times$ denotes the Cartesian product.
\end{proof}

\begin{figure}
\centering
\includegraphics[scale=.7]{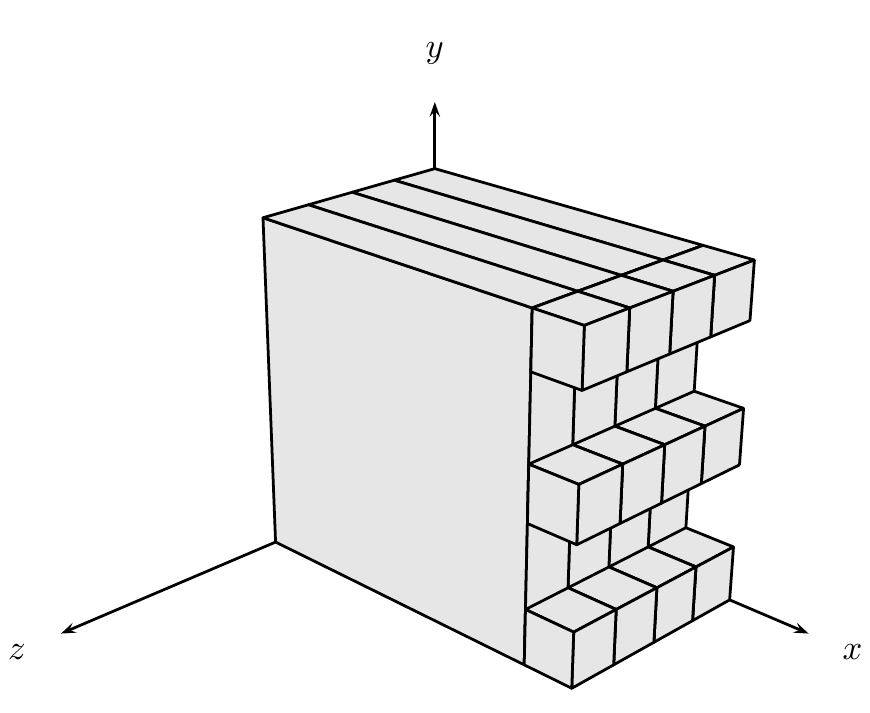}
\caption{An intersection representation of $T_3 \boxtimes P_4$.}
\label{fig:PT}
\end{figure}

Recall that the \textit{diameter} of a convex shape is the largest distance between any two points contained in it. We will use the fact that for any convex shape of diameter $\rho$, there is a ball of radius $\rho$ containing it. We will also use the following correspondence between aspect ratio and fatness; recall that a convex shape $Q$ is \textit{$k'$-fat} if every ball $B$ with center in $Q$ satisfies $\vol(B \cap Q) \geq \min \left( \vol(B)/k', \vol(Q) \right)$. From Lemma~\ref{lemma-height} and Lemma~2.4 of van der Stappen, Halperin, and Overmars~\cite{vanDerStappen1993}, we obtain the following.

\begin{lemma}
\label{lem:AR}
For all positive integers $d$ and $k$, there exists an integer $k'$ so that every convex shape in $\bb{R}^d$ with aspect ratio at most $k$ is $k'$-fat.
\end{lemma}

We are ready to prove the main lemma of this section.

\begin{lemma}
For all positive integers $d$, $c$, and $k$, there exist positive integers $r$ and $t$ so that $T_r \boxtimes P_t$ does not admit a $c$-thin $S$-shaped intersection representation,
where $S$ is the set of all convex shapes in $\bb{R}^d$ with aspect ratio at most $k$.
\end{lemma}
\begin{proof}
Let $k'$ denote the integer from Lemma~\ref{lem:AR}. Define\begin{align*}
    t &= ck'(6kd)^d+1, \\
    r_i&=\begin{cases}
            1&\text{if }i=t,\\
            ck'(t+i+1)^d(2kd)^d+r_{i+1}& \text{if }i\in\{0,1,\ldots,t-1\}\textrm{, and }
            \end{cases}\\
    r &= r_0 .
\end{align*}
\noindent Suppose for a contradiction that $T_r \boxtimes P_t$ has a $c$-thin $S$-shaped intersection representation $\varphi$.

Let $\pi_P:V(T_r \boxtimes P_t)\to V(P_t)$ and $\pi_T:V(T_r \boxtimes P_t)\to V(T_r)$ be the projections to the product factors.
Let $c$ be the center of the star $T_r$ and let $X=\pi_T^{-1}(c)$ denote the set of the $t$ vertices of $T_r \boxtimes P_t$ of maximum degree.
Let $x$ be a vertex in $X$ with the diameter $\rho$ of $\varphi(x)$ minimum. There is a ball $B$ of radius $\rho$ which contains $\varphi(x)$; by translating we may assume that $B$ is centered at the origin.
Let $v_1$, \ldots, $v_t$ be the vertices of $P_t$ in order.  For $i\in\{1,\ldots,t\}$, let $Y_i=\pi_P^{-1}(v_i)\setminus X$ be the set of vertices
of $T_r \boxtimes P_t$ corresponding to $v_i$ and to the rays of $T_r$. Let $a$ be the index such that $v_a=\pi_P(x)$.

For $i\in\{1,\ldots,t\}$, let $Y'_i\subseteq Y_i$ consist of the vertices $u\in Y_i$ such that $\varphi(u)$ has
a non-empty intersection with $\left(1+\frac{|i-a|}{t}\right)B$.
We will show by induction on $b=|i-a|$ that $|Y'_i|\ge r_b$ for $i\in\{1,\ldots,t\}$.
Certainly this is true for $b=0$; now suppose for a contradiction the claim holds for $b$ but not $b+1$;
by symmetry, it suffices to consider the case that $i>a$.  Each vertex $v\in Y'_i$ has a unique neighbor $v'\in Y_{i+1}$,
and since $\varphi(v)$ intersects $\left(1+\frac{b}{t}\right)B$, if $\varphi(v')$ is disjoint from $\left(1+\frac{b+1}{t}\right)B$,
then $\varphi(v)$ must have diameter at least $\rho/t$.  Consequently,
more than $r_{b}-r_{b+1}$ vertices $v \in Y'_{i}$ are mapped to shapes of diameter at least $\rho/t$.
So the height of $\varphi(v)$ is at least $\tfrac{\rho}{tk}$, and by Lemma~\ref{lemma-height},
we conclude $\vol(\varphi(v)) \geq \vol(B)/(2tkd)^d$. So, by Lemma~\ref{lem:AR}, a ball $B_v$ of radius $\rho/t$ centered at
a point in $\varphi(v) \cap \left(\left(1+\frac{b}{t} \right)B\right)$ satisfies
$$
    \vol(B_v \cap \varphi(v)) \geq \min \left(\frac{\vol(B)}{k't^d},\frac{\vol(B)}{(2tkd)^d} \right) \geq \frac{\vol(B)}{k'(2tkd)^d}.
$$
\noindent Thus, as $B_v \subseteq \left(1+\frac{b+1}{t}\right)B$ and $\varphi$ is $c$-thin, summing over all such vertices $v\in Y_i$, we obtain 
$$
    c\left(1+\frac{b+1}{t}\right)^d\vol(B) > \frac{(r_b-r_{b+1})\vol(B)}{k'(2tkd)^d}.
$$\noindent This is a contradiction to the definition of $r_b$. 
Therefore, $|Y'_{i+1}|\ge r_{b+1}$.

One final application of the same argument shows that every vertex $z \in X$ is mapped to a shape $\varphi(z)$ with non-empty intersection with $2B$.
By the choice of $x$, $\vol(\varphi(z))\geq \vol(\varphi(x))$.  Since $\varphi(x)$ has diameter $\rho$, it has height at least
$\rho/k$, and by Lemma~\ref{lemma-height} we conclude $\vol(\varphi(x))\geq\vol(B)/(2kd)^d$.
Lemma~\ref{lem:AR} implies a ball $B_z$ of radius $\rho$ centered at a point in $\varphi(z) \cap (2B)$ satisfies
$$
    \vol(B_z \cap \varphi(z)) \geq \min \left(\frac{\vol(B)}{k'},\frac{\vol(B)}{(2kd)^d} \right) \geq \frac{\vol(B)}{k'(2kd)^d}.
$$
Then, as each of these balls $B_z$ is contained in $3B$, we obtain
$$
    c3^d\vol(B) \geq \frac{t\vol(B)}{k'(2kd)^d},
$$
which is a contradiction to the definition of $t$. This completes the proof.
\end{proof}

\section{Non-representable example}\label{sec-converse}

This section is occupied by the proof of Theorem~\ref{thm-converse}.
 
For a box $B=I_1 \times I_2 \times \ldots \times I_d \subseteq \bb{R}^d$ and $1 \leq i \leq d$, let $\pi_i(B)=I_i$ denote the projection of $B$ on to the $i$-th coordinate axis.   We start with an easy and standard  lemma.
\begin{lemma}\label{lemma-boxes} Let $\mc{B}$ be a collection of boxes in $\bb{R}^d$, and let $n=|\mc{B}|$. Then for any positive integer $k$, either \begin{description}
		\item[(i)] there exists $\mc{B}' \subseteq \mc{B}$ with $|\mc{B'}|=k$ and $i\in\{1,\ldots,d\}$ such that the projections $\{\pi_i(B)\}_{B \in \mc{B}'}$ are pairwise disjoint, or
		\item[(ii)] there exists $\mc{B}' \subseteq \mc{B}$ with $|\mc{B'}| \geq n/k^d$ such that $\cap_{B \in \mc{B}'}B \neq \emptyset$
	\end{description}
\end{lemma}	
\begin{proof}
	The proof is by induction on $d$. For the base case, note that the intersection graph of intervals in $\bb{R}$ is perfect, and so if (i) does not hold then there exists   $\mc{B}' \subseteq \mc{B}$ with $|\mc{B'}| \geq n/k$ such that the intervals in $\mc{B}'$ have pairwise non-empty intersections. By Helly's property, it follows that $\mc{B}'$ satisfies (ii).
	
	For the induction step, for every $B \in \mc{B}$, let $\psi(B) \subseteq \bb{R}^{d-1}$ be such that $B=\psi(B) \times \pi_d(B)$. By the induction hypothesis applied to $\{\psi(B)\}_{B \in \mc{B}}$ either (i) holds, or there exists $\mc{B}'' \subseteq \mc{B}$ with $|\mc{B}''| \geq n/k^{d-1}$ such that  $\cap_{B \in \mc{B}''}\psi(B)  \neq \emptyset$. Applying the base case to $\mc{B}''$, we deduce that the lemma holds for $\mc{B}$.
\end{proof}

We also need the following lemma on interval representations.
For a closed interval $I = [a+b,a-b]$ and $k > 0$, let $k * I = [a+kb, a-kb]$, and let $|I|=2b$ be the length of $I$.
\begin{lemma}\label{lemma-combip}
Let $\mc{U}$ and $\mc{V}$ be finite sets of pairwise disjoint closed intervals of non-zero length,
and let $l$ and $s'$ be positive integers such that
$I \cap (l * J) \neq \emptyset$ and $|J|\le s'|I|$ for every $I\in \mc{U}$ and $J\in \mc{V}$.
If $|\mc{U}|\ge s'+6$, then $|\mc{V}|\le 2s'l^2$.
\end{lemma}
\begin{proof}
Let $\mu = \min_{I \in \mc{U}}|I|>0$.  Note that $|J|\le s'\mu$ for every $J\in \mc{V}$.
Since the intervals in $\mc{U}$ are pairwise disjoint, every interval in $\mc{V}$ intersects at most $s'+ 2$ intervals in $\mc{U}$. 

Let $I'$ and $I''$ be the second from the left and the second from the right intervals in  $\mc{U}$, respectively. As $|\mc{U}| \geq  s'+6$,
at least $s'+2$ intervals in $\mc{U}$ lie in between $I'$ and $I''$.
Therefore every interval in $\mc{V}$ either lies strictly to the right of $I'$, or strictly to the left of $I''$.
Without loss of generality we assume that there exists $\mc{V}' \subseteq \mc{V}$ such that $|\mc{V}'| \geq |V|/2$,
and every interval in $\mc{V}'$ lies strictly to the right of $I'$.
Let $I$ be the leftmost interval in $\mc{U}$. Then $I'$ separates $I$ from every interval in $\mc{V}'$,
and since $I \cap (l * J) \neq \emptyset$, we have $|J| \geq |I'|/l \geq \mu/l$ for every $J \in \mc{V}'$.
Finally, let $J_0$ be the rightmost interval in $\mc{V}'$. Since $I \cap (l * J_0) \neq \emptyset$,
we have $$l|J_0| \geq \sum_{J \in \mc{V}'} |J| \geq  \frac{|\mc{V}|}{2} \cdot \frac{\mu}{l},$$
and thus $|\mc{V}|\le 2l^2\tfrac{|J_0|}{\mu}\le 2l^2s'$.
\end{proof}

We continue by describing a construction of the graph that will satisfy Theorem~\ref{thm-converse}.
For a graph $G$, and positive integers $N$ and $l$, let the graph $G^{+}=G^{+}(N,l)$ be obtained from $G$ by adding vertices $v_1,v_2,\ldots,v_N$ and pairwise internally vertex disjoint paths $\{P_{i,u}\}_{i \in [N],u \in V(G)}$ each of length $l$, such that $P_{i,u}$ has ends $v_i$ and $u$ for every $i \in [N],u \in V(G)$. We say that $v_1,\ldots,v_N$ are the \emph{hubs} of $G^{+}$, and the $N$ connected components of $G^{+} \setminus V(G)$ are the \emph{stars} of $G^{+}$. Let $N_0 ,N_1, \ldots, N_d$ and $l_1,\ldots,l_d$ be positive integers. Let $G_0$ be an edgeless graph on $N_0$ vertices, let $G_i=G^{+}_{i-1}(N_i,l_i)$ for $1 \leq i \leq d$, and let $G=G_d$.

We show that if $N_0 \ll l_1 \ll N_1 \ll l_2 \ll \ldots l_d \ll N_d$ are sufficiently large with respect to 
parameters $c,s,d$ and $f$ of Theorem~\ref{thm-converse} then $G$ satisfies the conditions (i) and (ii).
The property (i) follows by a repeated application of the following claim (noting that $|V(G_{i-1})|$ only depends
on $N_0, l_1, N_1, \ldots, N_{i-1}$, and thus $l_i$ can be chosen sufficiently large compared to $|V(G_{i-1})|$ for $i=1,\ldots, d$).
\begin{lemma}\label{lemma-itercol}
Let $f: \bb{N} \to [3,+\infty)$ be a non-decreasing function such that $\lim_{r\to \infty}f(r) = +\infty$.
Suppose $G$ is a graph and $\prec$ is a linear order $\prec$ on $V(G)$ such that $\col_{\prec,r}(G)\le f(r)$ holds
for every non-negative integer $r$.  Then there exists $l_0$ such that for every positive integer $N$ and for every $l\ge l_0$,
the ordering $\prec$ can be extended to $G^+=G^+(N,l)$ so that $\col_{\prec,r}(G^+)\le f(r)$ holds for every non-negative integer $r$.
\end{lemma}
\begin{proof}
Choose $l_0$ so that $f(r)\ge |V(G)|$ for every $r\ge l_0$.
We extend $\prec$ so that the vertices of $G$ precede all other vertices of $G^+$, and moreover, the hubs of $G^+$ precede
the remaining vertices in $V(G^+)\setminus V(G)$.
If $r<l$, then $|L_{G^+,\prec,r}(v)|=|L_{G,\prec,r}(v)|\le f(r)$ for every $v\in V(G)$
and $|L_{G^+,\prec,r}(v)|\leq 3$ for every $v\in V(G^+)\setminus V(G)$.  If $r\ge l\ge l_0$,
then $|L_{G^+,\prec,r}(v)|\le \max(3,|V(G)|)\le f(r)$ for every $v\in V(G^+)$.
\end{proof}

It remains to show that (ii) holds. Suppose for a contradiction that there exists a $(c,\sqsubseteq_s)$-tame representation $\varphi$ of $G$ in $\bb{R}^d$ such that $\varphi(v) = p(v) + B(v)$  for some $p(v) \in \bb{R}^d$ and  some centrally symmetric box $B(v)$. It follows from  Lemmas~\ref{lemma-rel2} and~\ref{lemma-cmp}, that there exists $s'$ dependent on $d$ and $s$ only such that $B(v) \subseteq s'B(u)$ for all $u,v \in V(G)$ satisfying $\vol(B(v)) \leq \vol(B(u))$. 

Let $V_{i,1},\ldots, V_{i,N_i}$ be the stars of $G_i$, and let $u_{i,j} \in V_{i,j}$ be chosen to maximize $\vol(B(u_{i,j}))$ for every $1 \leq j \leq N_i$. Let $z_i \in V(G_{i-1})$ be chosen so that $\vol(B(z_i))$ is minimum. Let $$J_i =\{ j \in [N_i] \:|\: \vol(B(u_{i,j})) > \vol(B(z_i))\}.$$  As in Theorem~\ref{thm-color}, 
let $\prec$ be a linear ordering of vertices $v\in V(G)$ in a non-increasing order according to the volume of $B(v)$.  Note that for each $j \in J_i$ there exists a path of length at most $2l_i$ with one end $z_i$ the other end  $u_{i,j}$ and all internal vertices in $V_{i,j}$. Thus, $|L_{\prec, 2 l_i}(z_i)| \geq |J_i|$. By Theorem~\ref{thm-color}, there exists $\delta=\delta(c,s,d) >0 $ such that  $|L_{\prec, 2 l_i}(z_i)| \leq \delta (2l_i)^d$. We assume that $N_i \geq 2\delta(2l_i)^d$, and so without loss of generality we may assume that $j  > N_i/2 $ for every $j \in J_i$. Let $W_i = \{u_{i,j}\}_{1 \leq j \leq N_i/2}$ for every $1 \leq i \leq d$, and let $W_0=V(G_0)$. 

By Lemma~\ref{lemma-boxes}, the pigeonhole principle, and the assumption that $\varphi$ is $c$-thin, it follows, given that $N_0,\ldots,N_d$ are chosen appropriately large, that there exist
$0 \leq i <  j \leq d$, $U \subseteq W_i$, $V \subseteq W_j$, and a linear projection $\pi: \bb{R}^d \to \bb{R}$
such that $|U|\geq s'+6$, $|V| > 32(s')^3 l^2_j$, the projections $\mc{U}=\{\pi( p(u) +B(u)):u \in U\}$ are pairwise disjoint, and so are
the projections $\mc{V}=\{\pi( p(v) +B(v)):v \in V\}$.  Since $i<j$, we have
$\vol(B(v)) \leq \vol(B(u))$ for $u\in U$ and $v\in V$, and thus $B(v)\subseteq s'B(u)$.
It follows that $|J|\le s'|I|$ for every $I\in\mc{U}$ and $J\in\mc{V}$.

Consider any $u \in U$ and $v \in V$.
There exists a path $P$ of length at most $2l_j$ from $u$ to $v$ such that $\vol(B(x)) \leq \vol(B(v))$ 
for every $x \in V(P)-\{u\}$. Recall that this implies $B(x) \subseteq s'B(v)$ for every such $x$. Observation~\ref{obs:central} implies, as in
the proof of Theorem~\ref{thm-color}, that $p(u)+B(u)) \cap (p(v)+lB(v)) \neq  \emptyset$, where
$l=4s'l_j$.  Consequently, $I \cap (l * J) \neq \emptyset$ for every $I\in\mc{U}$ and $J\in\mc{V}$.
However, $|\mc{V}|>32(s')^3 l^2_j=2s'l^2$, which contradicts Lemma~\ref{lemma-combip}.
This finishes our proof of Theorem~\ref{thm-converse}.

\section{Concluding remarks}\label{sec-rem}

Note that in the proof of Theorem~\ref{thm-color} via Lemma~\ref{lemma-color}, the assumption that the
representation is $c$-thin is only used to lower-bound the volume of $X_v$.
Hence, this assumption can be relaxed somewhat; we did not do so as the
relaxations we could come up with are rather technical and restrict the placement
of the sets representing the vertices rather than just the shapes.

While we argued that the assumption that the shapes in $S$ are
pairwise $\le_k$-comparable for some fixed $k$ is necessary, the assumption they are
pairwise $\sqsubseteq_s$-comparable for a fixed $s$ is not.
For example, let $\ell_1=1$ and $\ell_{h+1}=\tfrac{1}{2(h+1)}\ell_h$ for all integers $h\ge 1$,
and consider the set $S^\star$ of shapes in $\mathbb{R}^2$ containing, for each positive integer $h$,
the trapezoid $T_h$ with vertices $(0,0)$, $(\ell_h,0)$, $(\ell_h,h\ell_h)$, and $(0,2h\ell_h)$
and an axis-aligned square $S_h$ with sides of length $\ell_h$.
Note that that $T_{h+1}\le_1 S_h\le_1 T_h$; see Figure~\ref{fig:ST}.  Furthermore, $S_h$ and $T_h$ are $\sqsubseteq_s$-incomparable
for every $s<\tfrac{3}{2}h$, and thus the elements of $S^\star$ are not pairwise $\sqsubseteq_s$-comparable
for any real number $s$.  However, thin $S^\star$-shaped representability does imply polynomial coloring numbers.

\begin{figure}
\centering
\includegraphics[scale=.75]{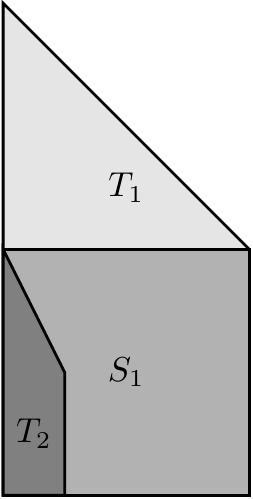}
\caption{The shapes $T_2$, $S_1$, and $T_1$.}
\label{fig:ST}
\end{figure}

\begin{lemma}
For every positive integer $c$, there exists a function $f_c\in O(r^2)$ such that
$\col_r(G)\le f_c(r)$ holds for every positive integer $r$ and every graph $G$ with a $c$-thin
$S^\star$-shaped intersection representation.
\end{lemma}
\begin{proof}
Note that if $h'\le h$, then $S_h\sqsubseteq_1 S_{h'}$, $T_{h+1}\sqsubseteq_{3/2} S_{h'}$, and $T_h\sqsubseteq_{3/2} T_{h'}$.
Hence, to bound $|L_{G,\prec,r}(v)|$, we can proceed in exactly the same way as in the proof of Theorem~\ref{thm-color}
if $\varphi(v)$ is a trapezoid, as well as to bound the number of vertices in $L_{G,\prec,r}(v)$ represented by squares in case $\varphi(v)$ is a square.
Hence, suppose that $\varphi(v)=S_h$ for some positive integer $h$, and let us bound the number of vertices $u\in L_{G,\prec,r}(v)$ represented by trapezoids $T_{h'}$
with $h'\le h$.  Let $Y_v$ be the square with sides of length $O(\ell_hr)$ as defined in the proof of Theorem~\ref{thm-color},
and let $D$ be the line segment forming lower side of $Y_v$.
The argument from the proof of Theorem~\ref{thm-color} can also be used to bound the number of vertices $u$ for which
$\varphi(u)$ does not intersect $D$, or for which $Y_v$ intersects $\varphi(u)$ anywhere outside of the
upper triangular part of $\varphi(u)$.  Furthermore, the ends of $D$ are contained in $\varphi(u)$ for at most $2c$ vertices $u$.

Let $Q\subseteq L_{G,\prec,r}(v)$ consist of the vertices $u$ such that $\varphi(u)$ is a trapezoid intersecting $Y_v$ only in its
upper triangular part, not containing the ends of $D$, and intersecting $D$.  As argued above, it suffices to bound the size of $Q$.
Note that for each $u\in Q$, there exists a translation of $T_h$ contained in $\varphi(u)$ and intersecting $D$ in its upper triangular part;
we let $T(u)$ be such a translation.  Let $p$ be a line parallel to $D$ at distance $h\ell_h$ below $D$.  Then $T(u)\cap p$ is a line segment
of length $\ell_h$ and contained in a segment of $p$ of length $O(\ell_hr)$.  Since $\varphi$ is $c$-thin, it follows that
$|Q|=O(cr)$.
\end{proof}
Hence, we leave open the question of whether there exists a natural sufficient and necessary
condition on a set $S$ of convex shapes enforcing that all graphs with a thin $S$-shaped intersection
representation have polynomial coloring numbers or sublinear separators.
Let us remark that no example of a class of graphs with strongly
sublinear separators (i.e., $O(n^{1-\varepsilon})$-separators for some fixed $\varepsilon>0$)
and superpolynomial coloring numbers is known.  Indeed, it has been asked~\cite{espsublin}
whether strongly sublinear separators imply polynomial coloring numbers.  If one is inclined
to think this is false, it might be natural to look for a counterexample in a geometric setting.

Considering Example~\ref{ex-narrow}, one could ask whether instead of giving
further assumptions on the shapes, it would perhaps be more natural to add the
assumption that for some fixed integer $m$, the represented graph does not
contain the complete bipartite graph $K_{m,m}$ as a subgraph.  This is the case
in $\mathbb{R}^2$, where the result of Lee~\cite{lee2016separators} implies
that for all positive integers $a\le b$, every intersection graph of connected
subsets of the plane avoiding $K_{a,b}$ as a subgraph has
$O\bigl(n^{1-\frac{1}{2a}}\bigr)$-separators.  However, the claim fails in $\mathbb{R}^3$, as for any
bipartite graph $H$, the graph $H'$ obtained by subdividing each edge once can
be represented as a $2$-thin intersection graph of pairwise $\le_1$-comparable
convex shapes: In Example~\ref{ex-narrow}, decrease the $z$-coordinate of the sets
$L_1$, \ldots, $L_m$ slightly to make them disjoint from $R_1$, \ldots, $R_m$,
and then add small cubes intersecting exactly $L_i$ and $R_j$ for each pair
$(i,j)$ corresponding to a vertex subdividing an edge of $H$.  If $H$ is a
$3$-regular expander, then $H'$ is an expander as well, and thus $H'$ does not
have sublinear separators; moreover, $K_{2,2}\not\subseteq H'$.

One might also ask why we only allow the translation of the shapes.
If we also allowed rotation, then by Lemma~\ref{lemma-incomp}, the existence of
sublinear separators is equivalent to all shapes having bounded aspect ratio.
Indeed, for every $k$ and $d\ge 2$, there exists $t$ such that every shape in $\mathbb{R}^d$ of aspect ratio at least $t$
has a rotation with which it is $\le_k$-incomparable.  The same is true if we instead allowed
scaling, up to a bijective affine transformation: by the affine transformation, we can assume
a cube is an envelope of one of the shapes $B\in S$, and if $S$ contains another shape of a large aspect
ratio, then we can scale it to be $\le_k$-incomparable with $B$ for a large $k$.

Although as argued above, the condition of $\sqsubseteq_s$-comparability is not necessary,
we consider it to be quite natural.  This leads us to a question of which graph classes
can be represented over pairwise $\sqsubseteq_s$-comparable sets of convex shapes;
we were not able to fully resolve this question, but let us mention several interesting observations.
Let us say that a class $\GG$ of graphs is \emph{$\sqsubseteq$-representable}
if there exist positive integers $d$, $c$, and $s$ such that for every graph $G\in\GG$, there exists
a set $S$ of pairwise $\sqsubseteq_s$-comparable convex shapes in $\mathbb{R}^d$ and a $c$-thin $S$-shaped
intersection representation of $G$.  Note that we allow the set of shapes to depend on $G$.
For example, Koebe's result~\cite{koebe} implies that the class of planar graphs is $\sqsubseteq$-representable.

The $\sqsubseteq$-representable classes enjoy a number of interesting closure properties.
Let us say that two shapes $B_1$ and $B_2$ are \emph{$\sqsubseteq_s$-equivalent}
if $B_1\sqsubseteq_s B_2$ and $B_2\sqsubseteq_s B_1$.  We say a class $\GG$ of graphs is \emph{$\square$-representable}
if there exist positive integers $d$, $c$, and $s$ such that for every graph $G\in\GG$, there exists
a set $S$ of pairwise $\sqsubseteq_s$-equivalent convex shapes in $\mathbb{R}^d$ and a $c$-thin $S$-shaped
intersection representation of $G$; as we will see below (Corollary~\ref{cor-sqrep}),
a class of graphs is $\square$-representable if and only if it has polynomial growth.
We say that $\GG$ is \emph{fat-$\square$-representable} if the condition of $c$-thinness is dropped.

For two classes $\GG_1$ and $\GG_2$ of graphs, let $\GG_1\boxtimes \GG_2=\{G_1\boxtimes G_2:G_1\in\GG_1,G_2\in\GG_2\}$,
where $\boxtimes$ is the strong product of graphs; and let $\GG_1\land \GG_2=\{G_1\cap G_2:G_1\in\GG_1,G_2\in\GG_2,V(G_1)=V(G_2)\}$.
\begin{lemma}\label{lemma-prod}
Let $\GG_1$ and $\GG_2$ be classes of graphs, where $\GG_1$ is $\sqsubseteq$-representable.
If $\GG_2$ is $\square$-representable, then $\GG_1\boxtimes \GG_2$ is $\sqsubseteq$-representable.
If $\GG_2$ is fat-$\square$-representable, then $\GG_1\land \GG_2$ is $\sqsubseteq$-representable.
\end{lemma}
\begin{proof}
For $i\in\{1,2\}$, let $\varphi_i$ be an $S_i$-shaped intersection representation of a graph $G_i\in\GG_i$,
where $S_i$ is a set of pairwise $\sqsubseteq_{s_i}$-comparable convex shapes in $\mathbb{R}^{d_i}$;
furthermore, assume that $\varphi_1$ is $c_1$-thin, and that the shapes in $S_2$ are pairwise $\sqsubseteq_{s_2}$-equivalent.
Let $S=\{B_1\times B_2:B_1\in S_1,B_2\in S_2\}$ be a set of shapes in $\mathbb{R}^{d_1+d_2}$, and
note that the shapes in $S$ are pairwise $\sqsubseteq_{s_1s_2}$-comparable.

For $v_1\in V(G_1)$ and $v_2\in V(G_2)$, let $\varphi_\boxtimes(v_1,v_2)=\varphi_1(v_1)\times\varphi_2(v_2)$,
and if $V(G_1)=V(G_2)$, then let $\varphi_\land(v_1)=\varphi_1(v_1)\times\varphi_2(v_1)$.  Then $\varphi_\boxtimes$
is an intersection representation of $G_1\boxtimes G_2$, and $\varphi_\land$ is an intersection representation of $G_1\cap G_2$.
Furthermore, $\varphi_\land$ is $c_1$-thin, and if $\varphi_2$ is $c_2$-thin, then $\varphi_\boxtimes$ is $c_1c_2$-thin.
The claims of the lemma easily follow.
\end{proof}

Clearly, if a class $\GG$ is $\sqsubseteq$-representable, then the class of all induced subgraphs of graphs in $\GG$ also is
$\sqsubseteq$-representable.  Somewhat more surprisingly, the same is true if we consider subgraphs instead of induced subgraphs.
To see this, we need the following observation.
\begin{lemma}\label{lemma-delstar}
Let $\GG$ be a class of graphs, and for a positive integer $k$, let $\GG_k$ be the class of graphs obtained from graphs in $\GG$
by deleting edges of the union of $k$ star forests.  If $\GG$ is $\sqsubseteq$-representable, then $\GG_k$ is $\sqsubseteq$-representable.
\end{lemma}
\begin{proof}
By Lemma~\ref{lemma-prod}, it suffices to argue that the class of complements of star forests is fat-$\square$-representable.
Indeed, the complement of a star with $t$ rays can be represented in $\mathbb{R}^2$ by a unit square $Q_1$ and $t$ copies of its translation $Q_2$
by the vector $(1+\varepsilon, 0)$ for arbitrarily small $\varepsilon>0$.  The complement of a star forests can then be represented by first
taking this representation for each of its stars, then rotating the representations slightly around the point $(1+\varepsilon/2,1/2)$
so that any two squares from representations of distinct stars intersect.
\end{proof}

\begin{corollary}\label{cor:subgraph}
Let $\GG$ be a class of graphs, and let $\GG'$ be the class of all subgraphs of graphs in $\GG$.
If $\GG$ is $\sqsubseteq$-representable, then $\GG'$ is $\sqsubseteq$-representable.
\end{corollary}
\begin{proof}
By Theorem~\ref{thm-color}, there exists an integer $a$ such that $\col_2(G)\le a$ for every graph $G\in\GG$,
and by~\cite{zhu2009colouring}, $G$ has star chromatic number at most $a^2$, and thus $G$ can be expressed as a union of $a^2$ spanning star forests.
Thus, any subgraph of $G$ can be obtained from an induced subgraph of $G$ by deleting edges of the union at most $a^2$ spanning forests.
The claim then follows from Lemma~\ref{lemma-delstar}.
\end{proof}

\begin{corollary}\label{cor-sqrep}
A class of graphs $\GG$ is $\square$-representable if and only if there exists a polynomial bounding the growth of
every graph in $\GG$.
\end{corollary}
\begin{proof}
When $\GG$ is $\square$-representable, the argument from the proof of Theorem~\ref{thm-color}
can be used to bound $\Delta_r$ rather than $\col_{\prec,r}$.
Conversely, if $\GG$ has polynomial growth, Theorem~\ref{thm-KL} implies that for every graph $G\in\GG$, a supergraph $G'$ of $G$
has a $2$-thin representation by unit balls in a bounded dimension.
Lemmas~\ref{lemma-prod} and~\ref{lemma-delstar} can be adapted to show that the class of all subgraphs of graphs in a fixed
$\square$-representable class is also $\square$-representable, completing the proof.
\end{proof}

Another easy construction is an addition of an apex vertex (a vertex adjacent to all the vertices of the graph): such a vertex can be
represented by a sufficiently large ball.

Finally, let us note that if we do not insist on considering intersection graphs, it is natural to consider the graphs with representation
described in Lemma~\ref{lemma-color}.  This representation is related to the overlap representation studied in~\cite{miller1998geometric}.

\bibliographystyle{siam}

\bibliography{../data}

\end{document}